\documentclass[12pt]{extarticle}
\usepackage[utf8]{inputenc}
\usepackage{amsmath, amsthm, amssymb, color}
\usepackage[colorlinks=true,linkcolor=blue,urlcolor=blue]{hyperref}
\usepackage{graphicx}
\usepackage{blkarray}
\usepackage{multicol}
\usepackage{caption}
\usepackage[numbers,sort]{natbib}
\usepackage{mathtools}
\usepackage{enumerate}
\usepackage{verbatim}
\usepackage{tikz,tikz-cd,tikz-3dplot}
\usepackage{amssymb}
\usepackage{multirow}
\usetikzlibrary{matrix}
\usetikzlibrary{arrows}

\usepackage{algorithm}
\usepackage[noend]{algpseudocode}
\usepackage{caption}
\usepackage[normalem]{ulem}
\usepackage{makecell}
\usepackage{array}
\usepackage{url}
\usepackage{color}
\usepackage{subcaption}
\usepackage{hyperref}
\usepackage[para]{footmisc}
\usepackage[official]{eurosym}
\oddsidemargin \evensidemargin
\newtheorem{theorem}{Theorem}
\newtheorem{proposition}[theorem]{Proposition}

\theoremstyle{definition}
\newtheorem{definition}[theorem]{Definition}

\newtheorem{remark}[theorem]{Remark}

\numberwithin{theorem}{section}
\theoremstyle{plain}

\newcommand{\PP}{\mathbb{P}}
\newcommand{\RR}{\mathbb{R}}

\newcommand{\CC}{\mathbb{C} }
\newcommand{\ZZ}{\mathbb{Z}}

\newcommand\CCt{{\mathbb C}\{\!\!\{t\}\!\!\}}
\newcommand\RRt{{\mathbb R}\{\!\!\{t\}\!\!\}}
\newcommand\polymake{{\texttt{polymake }}}

\DeclareMathOperator{\trop}{trop}

\DeclareMathOperator{\val}{val}

\newtheorem{innercustomthm}{Theorem}
\newenvironment{customthm}[1]
{\renewcommand\theinnercustomthm{#1}\innercustomthm}
{\endinnercustomthm}

\usepackage{listings}
\usetikzlibrary{shapes.geometric, arrows}	

\tikzstyle{startstop} = [minimum width=3cm, text width = 5cm, minimum height=1cm, text centered, draw=black]
\tikzstyle{arrow} = [thick, ->, >=stealth]

\usetikzlibrary{matrix}
\usetikzlibrary{patterns}
\usetikzlibrary{positioning}
\usetikzlibrary{decorations.pathmorphing}
\usetikzlibrary{cd}

\usepackage{algorithmicx}
\usepackage{algorithm}
\usepackage{algpseudocode}

\MakeRobust{\Call}

\algdef{SE}[DOWHILE]{Do}{DoWhile}{\algorithmicdo}[1]{\algorithmicwhile\ #1}
\algrenewcommand{\algorithmiccomment}[1]{\hfill $\rhd$ \emph{#1}}
\algrenewcommand{\algorithmicrequire}{\textbf{Input:}}
\algrenewcommand{\algorithmicensure}{\textbf{Output:}}
\algnewcommand{\Or}{\textbf{or}}
\algnewcommand{\And}{\textbf{and}}
\algnewcommand{\Not}{\textbf{not}\,}
\algnewcommand\algorithmicforeach{\textbf{for each}}
\algdef{S}[FOR]{ForEach}[1]{\algorithmicforeach\ #1\ \algorithmicdo}
\def\NoNumber#1{{\def\alglinenumber##1{}\State #1}\addtocounter{ALG@line}{-1}}

\title{Real Tropical Quartics and their Bitangents: \\Counting with Patchworking}
\author{Alheydis Geiger }
\date{}

\begin{document}
	
	\maketitle
	\begin{abstract}
		For certain tropical quartic curves the existing techniques could not predict the lifting behavior of their bitangents over the real numbers. We close this gap     
		by using patchworking techniques. Further, this paper provides an analysis of the combinatorial types of real tropical quartic curves according to their real topology and number of real bitangents. This paper is accompanied by an extension for \texttt{polymake}; the computational data is available in the database collection \texttt{polyDB}.
	\end{abstract}
	
	\section{Introduction}
	
	Real bitangents of quartic curves have been studied in algebraic geometry for two centuries \cite{Zeu73,Plue34,Plue39,PlStVi11,He1855}.
	When tropical geometry and its interactions with algebraic geometry developed, the \emph{lifting} of tropical curves over different fields, i.e., finding algebraic curves that tropicalize to the given tropical curve, became more of a focus. 
	Next to classical counting problems like Gromov-Witten and Welschinger invariants \cite{Mik, IKS03}, relative lifting questions were prominent: Given a tropical variety $X$ with a fixed algebraic lift $C$ and another tropical variety $Y$ that satisfies some relation to $X$, can we find an algebraic variety $Z$ that is a lift of $Y$ and satisfies the same relation to $C$? Well-known such relative lifting problems are lines on a cubic surface \cite{Vig10,PaVi19,BK12,BS15,G20} or bitangents of a quartic curve \cite{LeMa19,LeLe18,1GP21}. The particular interest in these questions arises from the superabundance phenomenon in tropical geometry: There can be infinitely many tropical lines on a tropical cubic surface or tropical bitangents to a tropical quartic curve.
	
	Len and Markwig \cite{LeMa19}  first investigated the lifting of tropical bitangents to quartics over $\CC$ and discovered that after grouping together into 7 so-called \emph{bitangent classes}, each class had 1, 2 or 4 representatives that lifted with multiplicity 4, 2 or 1, respectively. Thus, they recovered Pl\"ucker's classical count of $28$ complex bitangents to a quartic \cite{Plue34,Plue39}. 
	Cueto and Markwig \cite{CueMa20} then classified the different combinatorial types of bitangent classes and their concrete lifting conditions over the real numbers in terms of sign conditions on the coefficients of the polynomial defining the algebraic quartic curve. Panizzut and the author \cite{1GP21} built on these results and developed the concept of \emph{deformation motifs}, which allowed determining the number of lifting bitangents directly from the dual subdivision of the tropical quartic and a given sign distribution.
	They implemented this in the extension \texttt{TropicalQuarticCurves} for \texttt{polymake} \cite{2GP21,polymake:2000}.
	
	The underlying techniques used in \cite{LeMa19,CueMa20} need a specific genericity constraint on the tropical quartic curve; see Definition \ref{def:generic}(i). This constraint is of vital importance in determining the lifting conditions for the bitangent class of shape (C), as depicted in Figure \ref{fig:shapeC}. However, it is not always satisfied for a generic point in the interior of a maximal cone in the secondary fan of tropical quartics. For these, it was so far not possible to determine the number of lifting bitangents.
	This paper closes this last remaining gap in our understanding of tropical bitangents of smooth quartic curves.

	\begin{customthm}{A}[Theorems \ref{thm:lift1}, \ref{thm:lift}, Proposition \ref{prop:totallyreal}]\label{thm:A}
		The sign conditions determined in \cite{CueMa20} are valid for all smooth tropical quartics, even if they are non-generic w.r.t. \ref{def:generic}(i). Further, for a generic tropical quartic all real bitangents are totally real, while for smooth tropical quartics non-generic w.r.t. \ref{def:generic}(i) at most one real bitangent is not totally real.
	\end{customthm}
	
	In order to achieve this result, the connection between the number of real bitangents and the topological type of the real curve was used. 
	The question for the possible topological types of smooth real curves of a given degree, was popularized by Hilbert, who included a version of it as the 16th problem in his famous list of mathematical problems. For real smooth quartic curves the answer goes back to Zeuthen, who connected it with the number of real bitangents~\cite{Zeu73}:  $n$ non-nested ovals for $n\in\{1,\ldots,4\}$ with $4,8,16,$ or $28$ real bitangents respectively, or $2$ nested ovals  with $4$ real bitangents. 
	To prove Theorem \ref{thm:A}, we use this classical connection and applied Viro's patchworking technique \cite{Viro} to allow a combinatorial and computational approach. 
	
	Furthermore, we use this approach to collect more data on real tropical quartic curves and their bitangents. This was done by implementing an update to the \texttt{polymake} extension \texttt{TropicalQuarticCurves} \cite{2GP21}, which originally accompanied \cite{1GP21}.
	
	\begin{customthm}{B}[Theorem \ref{thm:main}]
		Table \ref{tab:main} shows the distribution of combinatorial types of real tropical quartic curves (Definition \ref{def:real}) according to their number of real bitangents and real topological type. 
		\begin{table}[h]
			\centering
			\begin{tabular}{c|c|c|c|c|c}
				\# ovals & 1 & 2 nested & 2 & 3 & 4\\ \hline
				\# real bitangents  & 4 & 4 & 8 & 16 &28  \\ \hline \hline
				\multirow{ 2}{*}{\#  comb. real trop. quartics} &  \multicolumn{2}{c|}{$47\,447\,552$} & \multirow{ 2}{*}{$39\,316\,992$} &  \multirow{ 2}{*}{$24\,961\,536$} & \multirow{ 2}{*}{$9\,875\,968$ }
				\\ 
				&  $34\,899\,968$ & $12\,547\,584$&  & & \\ \hline 
				\#  comb. real trop. quartics
				&    \multicolumn{2}{c|}{$ 8\,165\,632$}  & \multirow{ 2}{*}{$6\,771\,456$} &\multirow{ 2}{*}{$4\,294\,912$} &\multirow{ 2}{*}{$1\,706\,752$}\\  
				mod $S_3$  & $6\,003\,712$ &$2\,161\,920$ & & & \\
			\end{tabular}
			\caption{The number of combinatorial types of real tropical quartics $(\mathcal{T},\delta)$ distributed with respect to their number of real bitangents and real ovals. The last row shows the count modulo the $S_3$ action, while the other row shows a total count.}
			\label{tab:main}
		\end{table}

	\end{customthm}
	
	The computation of the data in Table \ref{tab:main} was expensive with respect to computing time: For each of the $1278$ unimodular triangulations of $4\Delta_2$ the number of bitangents and real ovals had to be computed for all $2^{14}$  possible sign distributions. The duration of this computation for just one triangulation is around 73 minutes. For more details, see Section \ref{sec:ext}. Adherent to the FAIR data principles \cite{FAIR}, the data leading to Theorem \ref{thm:main} are stored for future use in the database collection \texttt{Tropical.QuarticCurves} on \texttt{polyDB} \cite{polydb:paper}.

	This paper is organized as follows: Section \ref{sec:pre} 
	gives a brief overview over the results on real tropical curves from the theory of patchworking. Section \ref{sec:results} first revisits some of the results on the lifting of bitangents on smooth tropical quartic curves from previous papers \cite{CueMa20,1GP21}. Then it presents the results regarding the lifting of tropical bitangents for non-generic smooth tropical quartic curves. Section \ref{sec:ext} showcases the additions to the \texttt{polymake} extension \texttt{TropicalQuarticCurves} and to the collection \texttt{Tropical.QuarticCurves} in \texttt{polyDB}.
	\medskip

	\noindent\textbf{Acknowledgements.} The author thanks Hannah Markwig, Michael Joswig, Johannes Rau, and Kris Shaw for insightful discussions, Marta Panizzut and M\'at\'e Telek for comments on this article, and Andreas Paffenholz for help with the database collection.

	\section{Prerequisites}\label{sec:pre}
	A \emph{plane quartic curve} is the vanishing set of a bivariate polynomial of degree four
	\begin{align*}
		f(x,y) =\,\, &a_{00} + a_{10}x+ a_{01}y+ a_{20}x^2 + a_{11}xy+ a_{02}y^2 + a_{30}x^3 + a_{21}x^2y\\&
		+ a_{12}xy^2 + a_{03}y^3 + a_{40}x^4 + a_{31}x^3y+ a_{22}x^2y^2 + a_{13}xy^3 + a_{04}y^4.   
	\end{align*}
	We write $\Gamma = \trop(V(f))$ for its tropicalization. The base field over which we tropicalize is a real closed complete non-archimedean valued field $K_\RR$ and its algebraic closure $K$, e.g. the completion of $\RRt$ and $\CCt$.  Note, that generalizing results, like the count of bitangents, to the reals follows from  Tarski-Seidenberg's Transfer Principle \cite[Theorem
	1.4.2]{Basu:2006}.
	When tropicalizing curves we adhere to the $\max$-convention. In this paper, we only consider \emph{smooth} tropical quartic curves: the Newton polytope of $f$ is the 4-dilated 2-dimensional simplex $4\Delta_2 =\{(x_1,x_2)\in\ZZ_{\geq0}^2 | x_1+x_2\leq 4  \}$, and the dual subdivision of $4\Delta_2$ induced by the coefficients $-\val(a_{ij})$ of the tropical polynomial $\trop(f)$ is a \emph{unimodular triangulation}, i.e., each triangle in the subdivision has volume $\frac{1}{2}$. The dual subdivision of $\Gamma$ will be denoted by $\mathcal{T}$ or $\mathcal{T}(\Gamma)$, if clarification is necessary. The set of points in $\RR^{15}$ that induce the same subdivision $\mathcal{T}$ of $4\Delta_2$ is a relatively open cone called \emph{secondary cone}. 
	The combinatorial aspects of the tropical curve $\Gamma = \trop(V(f))$ are determined by the dual subdivision $\mathcal{T}(\Gamma)$.

	In \cite{Mik,Schroeter_2013,RS23,texier2021topologyrealalgebraiccurves} a real tropical curve is defined as a tuple of a non-singular tropical curve together with a real phase structure. This real phase structure can be defined directly on a non-singular curve, without a distribution of signs on the lattice points of the dual subdivision $\mathcal{T}(\Gamma)$. However, up to multiplication with $(-1)$, the real phase structure defines a sign distribution on the lattice points of $\mathcal{T}$; see \cite[Remark 3.8]{RS23}. Since we will only use the sign distribution but not the real phase structure, we use the following definition for a real tropical curve. For a definition of the real phase structure we refer to \cite[Definition~3.5]{texier2021topologyrealalgebraiccurves}.
	
	\begin{definition}\label{def:real} A \emph{real tropical curve} is a tuple $(\Gamma,\delta)$, where $\Gamma$ is a smooth tropical curve and $\delta\in \{\pm 1\}^{\#(\mathcal{T}\cap\ZZ^2)}$ is a distribution of signs on the lattice points of the dual subdivision~$\mathcal{T}$. 
		We say $(\mathcal{T},\delta)$ is the \emph{combinatorial type} of the real tropical curve. 
	\end{definition}
	The combinatorial type of a real tropical curve determines many properties of the curve, like its number of real ovals or real bitangents. The symmetry group $S_3$ acts on the combinatorial type by permuting the lattice points of $4\Delta_2$ by reflections along the triangle medians. The sign distribution is unique up to a multiplication by $-1$, so there are $2^{14}$ possible patterns. 
	As there are $1278$ unimodular triangulations of $4\Delta_2$ up to the $S_3$ action \cite[Table 1]{BJMS15}, there are $1278\cdot 2^{14}$ combinatorial types of real tropical curves up to $S_3$ symmetry.
	
	Now we explain how to connect real tropical curves and real algebraic curves. We identify $\ZZ_2^2$ with the reflections in $\RR^2$ with respect to the coordinate hyperplanes: For $\epsilon\in\ZZ^2_2$, we write $\epsilon(x,y) = ((-1)^{\epsilon_1}x,(-1)^{\epsilon_2}y ).$ Let $\mathbb{T}^2$ denote the tropical projective plane as in \cite[Section 6.2]{MS15}.
	The \emph{real tropical plane} 
	$\mathbb{T}\RR^2 $  is homeomorphic to $\bigsqcup_{\epsilon\in\ZZ^2_2} \epsilon(\Delta_2)/\sim$, where a face $F$ of the polytope is identified with $\epsilon(F)$ for $\epsilon = (\overline{\alpha_1},\overline{\alpha_2})$ where $(\alpha_1,\alpha_2)$ is the integer vector orthogonal to the primitive integer direction of $F$. This is illustrated in blue in Figure~\ref{fig:patchworking2}. The real part of a real tropical curve is a curve in the real tropical plane. We construct it from $(\Gamma,\delta)$ by the procedure below. Viro's patchworking theorem (see Theorem~\ref{thm:viro}) then verifies that this fits together with the real part of an algebraic lift of the tropical curve.
	
	For a given real tropical curve $(\Gamma,\delta)$ with Newton polytope $d\Delta_2$ we extend the sign distribution $\delta$ to the integer points $v$ of $\bigsqcup_{\epsilon\in\ZZ^2_2}\epsilon(d\Delta_2)$ by setting $\delta(\epsilon(v)) = (-1)^{\epsilon_1v_1+\epsilon_2v_2}\delta(v).$ This is shown in Figures \ref{fig:patchworking1} and \ref{fig:patchworking2}.
	
	For an edge $e$ of $\Gamma$ the \emph{real part} $\RR e_\delta$ of $e$ with respect to $\delta$ is 
	$$\RR e_{\delta}:= \bigcup_{\epsilon\in\ZZ^2_2 \text{ such that} \atop \text{the vertices of }\epsilon(e^\vee) \text{ have opposite signs}} \epsilon (e).$$
	
	\begin{definition}[{\cite[Definition 3.2]{texier2021topologyrealalgebraiccurves}}]
		The \emph{real part} of a smooth real tropical curve $(\Gamma,\delta)$ is  $$\RR\Gamma_{\delta}:= \overline{\bigcup_{e\in\text{Edge}(\Gamma)}\RR e_{\delta}}\subset \mathbb{T}\RR^2.$$ 
	\end{definition}
	Figure \ref{fig:patchworking3} shows a schematic drawing of the real edges corresponding to the given subdivision and sign distribution. A visualization  by \texttt{polymake} of the real part of a real tropical curve with combinatorial type as in Figure \ref{fig:patchworking1} is presented in Figure \ref{fig:patchworking4}.
	
	\begin{figure}
		\centering
		\begin{subfigure}[c]{.2\textwidth}
			\centering
			\includegraphics[width=0.75\linewidth]{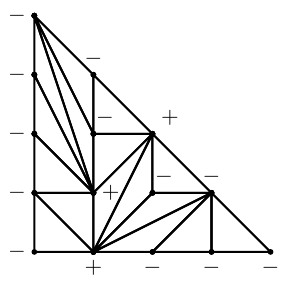}
			\caption{}\label{fig:patchworking1}
		\end{subfigure}  
		\begin{subfigure}[c]{.25\textwidth}
			\centering
			\includegraphics[width=0.95\linewidth]{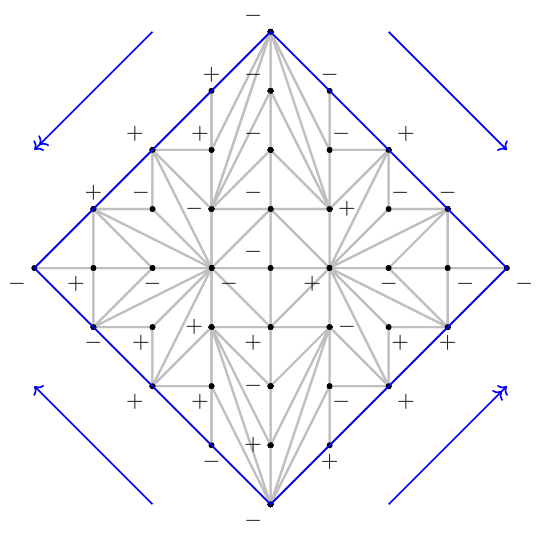}
			\caption{}\label{fig:patchworking2}
		\end{subfigure}
		\begin{subfigure}[c]{.25\textwidth}
			\centering
			\includegraphics[width=0.95\linewidth]{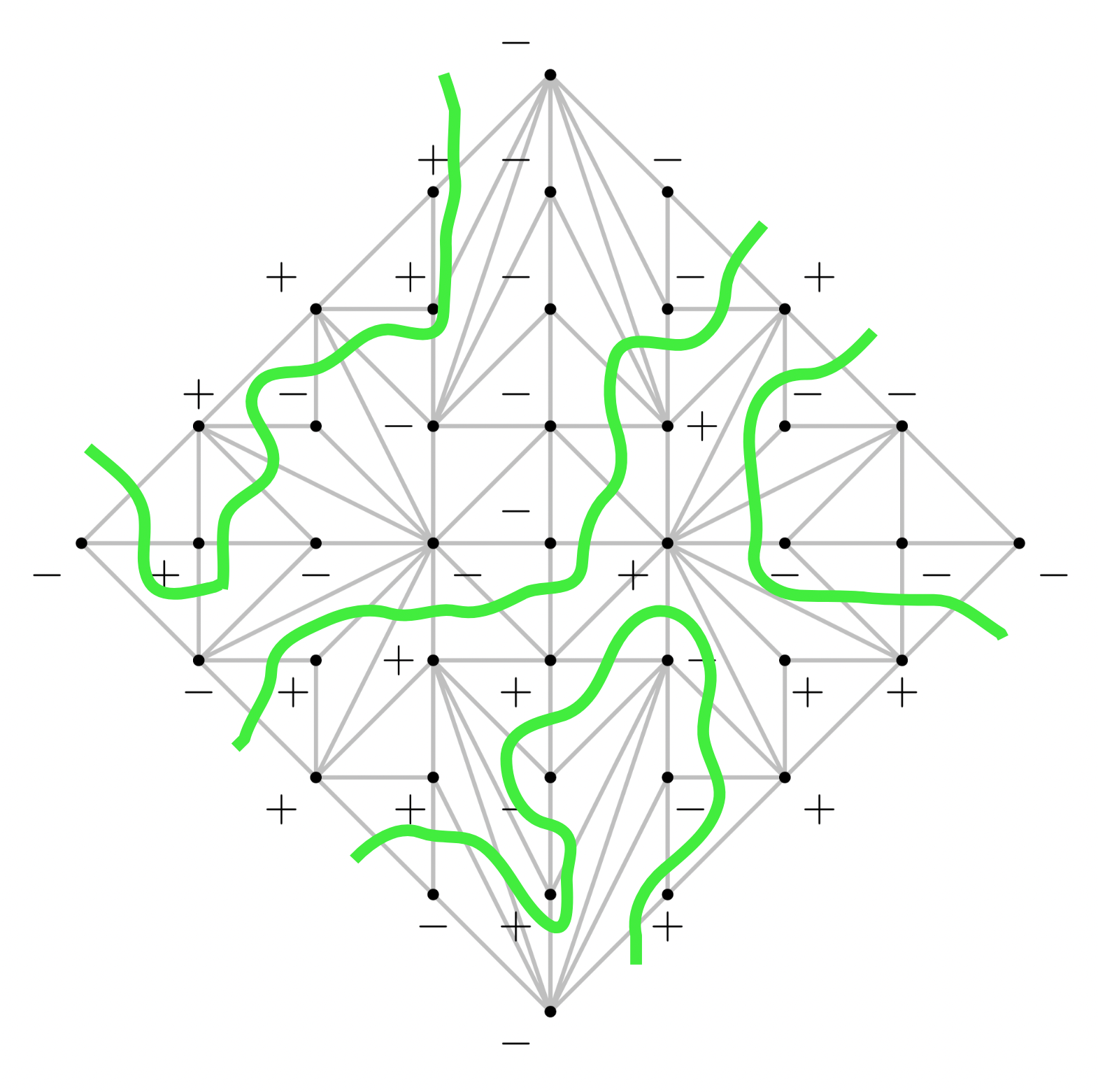}
			\caption{}\label{fig:patchworking3}
		\end{subfigure}  
		\begin{subfigure}[c]{.25\textwidth}
			\centering
			\includegraphics[width=0.95\linewidth]{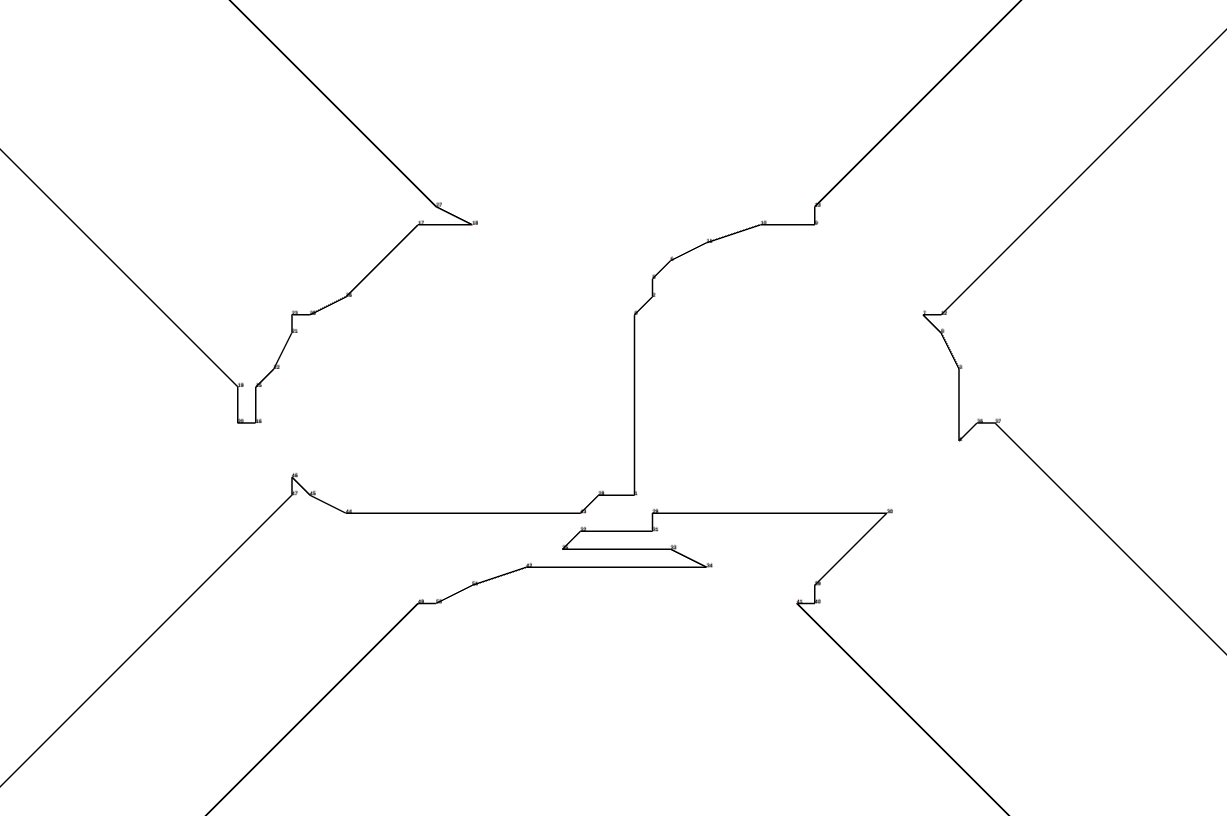}
			\caption{}\label{fig:patchworking4}
		\end{subfigure}
		\caption{Illustration of the patchworking technique. }
		\label{fig:patchworking}
	\end{figure}
	
	One tool to determine properties of real tropical curves is the patchworking technique~\cite{Viro}. We use Viro's patchworking theorem as presented in \cite{texier2021topologyrealalgebraiccurves}.
	\begin{theorem}[{\cite[Theorem 3.3]{texier2021topologyrealalgebraiccurves}, \cite{Viro}}]\label{thm:viro}
		Let $(C_t)_t$ be a family of
		non-singular real algebraic plane curves
		converging to a non-singular tropical curve $\Gamma$ (in the sense of Remark~\ref{rem:converge}), such
		that $\Gamma$ has Newton polygon $d\Delta_2$, where $d$ is the degree of $C_t$. Let $(P_t)_t$ be a family of polynomials
		defining $(C_t)_t$, and let $\delta$ be the distribution of signs of the coefficients. For $t \in\RR_{>0}$ small enough, there exists
		a homeomorphism from $ \PP^2_\RR \to \mathbb{T}\RR^2$  
		which maps the set of real points $C_t(\RR)$ onto $\RR\Gamma_\delta$.
	\end{theorem}
	Thus, using  patchworking we can determine the topology of a real lift of a real tropical curve, i.e., the number of ovals and whether or not they are nested.

	\begin{remark}\label{rem:converge}
		A tropical plane curve can be approximated by algebraic curves in $(\CC^\times)^2$ using the logarithmic map $\text{Log}_t(z,w) = (\log_t |z|,\log_t |w|),$ where $t\in\RR_{>0}$. The image of a curve $C$ under this map is called the \emph{amoeba} of $C$ in base $t$.
		We say a collection of algebraic curves $(C_t)_t$ approximates a tropical plane curve $\Gamma$, if \begin{itemize}
			\item for every $t$ the polynomial $P_t$ defining the curve $C_t$ has coefficients that are functions $\RR\to\CC$ in $t$,
			\item the vanishing order of the coefficients for $t\in\RR_{>0}$ tending to 0 agrees with the coefficients of the tropical polynomial defining $\Gamma,$ i.e. for $t\to 0$ we have $\alpha_{i,j}(t) \sim \beta_{i,j}t^{-a_(i,j)}$ for $\alpha_{i,j}(t)$ coefficient of $P_t$, $\beta_{i,j}\in\CC^\times$, $a_{i,j}\in\mathbb{R}\cup\{-\infty\}$ and $\Gamma$ is defined by the tropical polynomial $P(x,y)=\max_{i,j}(a_{i,j} + i x + jy).$
		\end{itemize}
		
	\end{remark}
	
	When a family of amoeba approximates a tropical curve $\Gamma$, there are two ways a bounded edge of $\Gamma$ can be approximated. See Figure \ref{fig:twist}. Edges that are approximated as in Figure~\ref{fig:twisted} are called \emph{twisted} \cite[Section 3]{BIMS15}. Not every edge of a tropical curve can be twisted. The potential edges that can be twisted, depending on the approximation, are called \emph{twist-admissible}.
	
	\begin{figure}
		\centering
		\begin{subfigure}{.4\textwidth}
			\centering\includegraphics[width=0.45\linewidth]{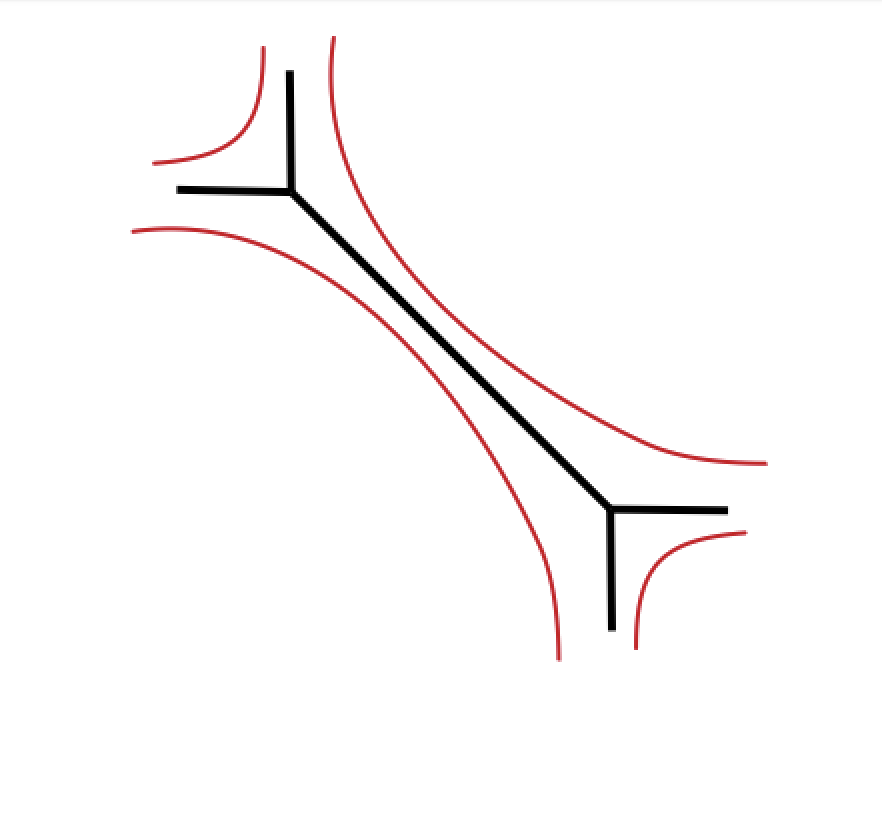} 
			\caption{Not twisted.}
		\end{subfigure}
		\begin{subfigure}{.4\textwidth}
			\centering\includegraphics[width=0.4\linewidth]{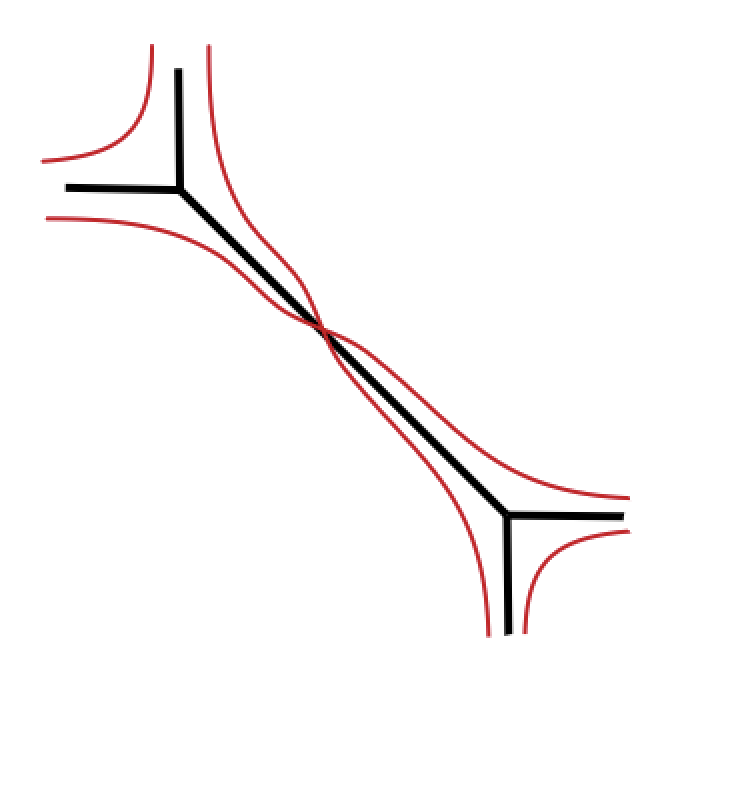}   
			\caption{Twisted.}\label{fig:twisted}
		\end{subfigure}
		\caption{Amoeba for a part of a tropical plane curve showing the difference between a twisted and a not twisted edge. }
		\label{fig:twist}
	\end{figure}

	\begin{definition}[{\cite[Definition 3.3]{BIMS15}}]
		A subset $T$ of the set of bounded edges of a tropical curve $\Gamma$ is called \emph{twist-admissible} if it satisfies the following condition.
		Consider $\Gamma$ as a metric graph and let $\gamma$ be a cycle of this graph. For any such cycle $\gamma$, take the edges in $\gamma\cap T =\{g_1,\ldots,g_k\}$. Let $(u_i,v_i)$ denote the primitive integer vector in the direction of $g_i$. Then
		$\sum_{i=1} ^k(u_i,v_i)=0 \mod 2.$
	\end{definition}
	
	As a consequence of how twisted edges arise from the family of approximating amoebas, the set of twisted edges of a tropical curve depends only on a given sign distribution $\delta$ on the coefficients \cite[Section 3.2]{BIMS15}.
	
	\begin{proposition}[{\cite[Remark 3.9]{BIMS15},\cite[Proposition 3.12]{texier2021topologyrealalgebraiccurves}}]
		Let $E$ be a bounded edge of a real tropical curve $(\Gamma,\delta)$. We denote by $e$ the edge in the dual subdivision $\mathcal{T}(\Gamma)$ that is dual to $E$. Let $v_1$ and $v_2$ be the two lattice points that span $e$. Since $\mathcal{T}(\Gamma)$ is unimodular, the edge $e$ is part of two triangles. We denote the two vertices of these triangles that are disjoint from $e$ as $a_1$ and $a_2$.
		
		\begin{itemize}
			\item If the coordinates modulo 2 of $a_1$ and $a_2$ 
			are distinct, then the edge $e$ is twisted if and only if $\delta(v_1)\delta(v_2)\delta(a_1)\delta(a_2) = +1.$
			\item If the coordinates modulo 2 of $a_1$ and $a_2$ coincide, then the edge $e$ is twisted if and only if $\delta(a_1)\delta(a_2) =-1.$
		\end{itemize}
		
	\end{proposition}

	Viro's patchworking theorem can be rephrased using twisted edges as follows.
	\begin{theorem}[{\cite[Theorem 3.4]{BIMS15}},\cite{Viro}]
		For any twist-admissible set $T$ in a non-singular tropical curve $C$ in
		$\RR^2$, there exists a family of non-singular real algebraic curves $(C_t)_{t\in \RR_{>0}}$ in $(\CC^\times)^2$ which
		converges to $C$ in the sense of Remark \ref{rem:converge}  and such that the corresponding set of twisted
		edges is $T$. 
	\end{theorem}

	In algebraic geometry, a complex curve $C$ is called \emph{dividing} if the real locus $C(\RR)$ of the curve divides the complex locus into two connected components. Otherwise, $C(\CC)\setminus C(\RR)$ is connected and the curve $C$ is not dividing.

	\begin{definition}[{\cite[Definition 4.6.]{texier2021topologyrealalgebraiccurves}}]
		A smooth real tropical curve $(\Gamma,\delta)$ is \emph{dividing}
		if $(\Gamma,\delta)$ is the tropical limit of a family of non-singular real algebraic curves $(C_t)_t$
		such that for $t\in R>0$ small enough, the curve $C_t$ is dividing.
	\end{definition}
	We will use the following theorem from \cite{Haa97} as presented in \cite{texier2021topologyrealalgebraiccurves} to determine whether a real tropical curve is dividing.
	\begin{theorem}[{\cite[Theorem 4.7]{texier2021topologyrealalgebraiccurves}},\cite{Haa97}]
		Let $(\Gamma,E)$ be a non-singular real tropical curve with admissible set of twisted edges $T$. Then $(\Gamma,E)$ is dividing if and only if for any cycle
		$\gamma$ in $\Gamma$ (seen as a subset of bounded edges of $\Gamma$), we have
		$$ |\gamma\cap T|= 0 \mod 2.$$
	\end{theorem}
	
	The goal of Section \ref{sec:results} is to determine the number of real bitangents as well as the topological type of the real part of the tropical curve, i.e. the number of ovals and whether or not they are nested.
	We use the following fact from the algebraic geometry of curves.
	\begin{proposition}\label{prop:dividing}
		Let $C$ be a quartic curve for which the real curve consists of $2$ ovals. The following are equivalent
		\begin{itemize}
			\item $C$ has $4$ real bitangents.
			\item $C$ is dividing.
			\item The $2$ real ovals are nested.
		\end{itemize}
	\end{proposition}
	\begin{proof}
		This follows from \cite[Corollary 5.3]{realcurves} and the arguments in the proof of \cite[Lemma~7.2]{realcurves}.
	\end{proof}

	\section{Counting Bitangents of smooth tropical quartics}\label{sec:results}
	
	A smooth tropical quartic curve can have infinitely many tropical bitangents. However, these can be grouped together in seven bitangent classes.
	In \cite{CueMa20} the shapes of these tropical bitangent classes were classified and, for generic curves, the real lifting conditions were determined. Later, the authors of \cite{1GP21} showed that the dual subdivision of a smooth tropical quartic curve determines the real lifting conditions of its bitangent classes, for which the different bitangent shapes were grouped into deformation classes with unique motifs in the dual subdivision.
	We briefly recall a few of the results of these articles.
	
	\begin{definition}[{\cite[Definition 3.1]{CueMa20}}]
		Let $\Gamma$ be a smooth tropical quartic curve and $\Lambda$ a tropical line that is bitangent to $\Gamma.$
		The \emph{tropical bitangent class} of $\Lambda$
		consists of the connected components of the subset of $\RR^2$ containing the vertices of all tropical bitangent lines linearly equivalent to $\Lambda$. The \emph{shape} of a tropical bitangent class refines each class by coloring those points belonging to the
		tropical quartic $\Gamma$, and subdividing edges and rays of a class accordingly. An example of a tropical curve with its bitangent classes is depicted in Figure~\ref{fig:shapeC}.
	\end{definition}

	\begin{figure}
		\centering
		\includegraphics[width=0.33\linewidth]{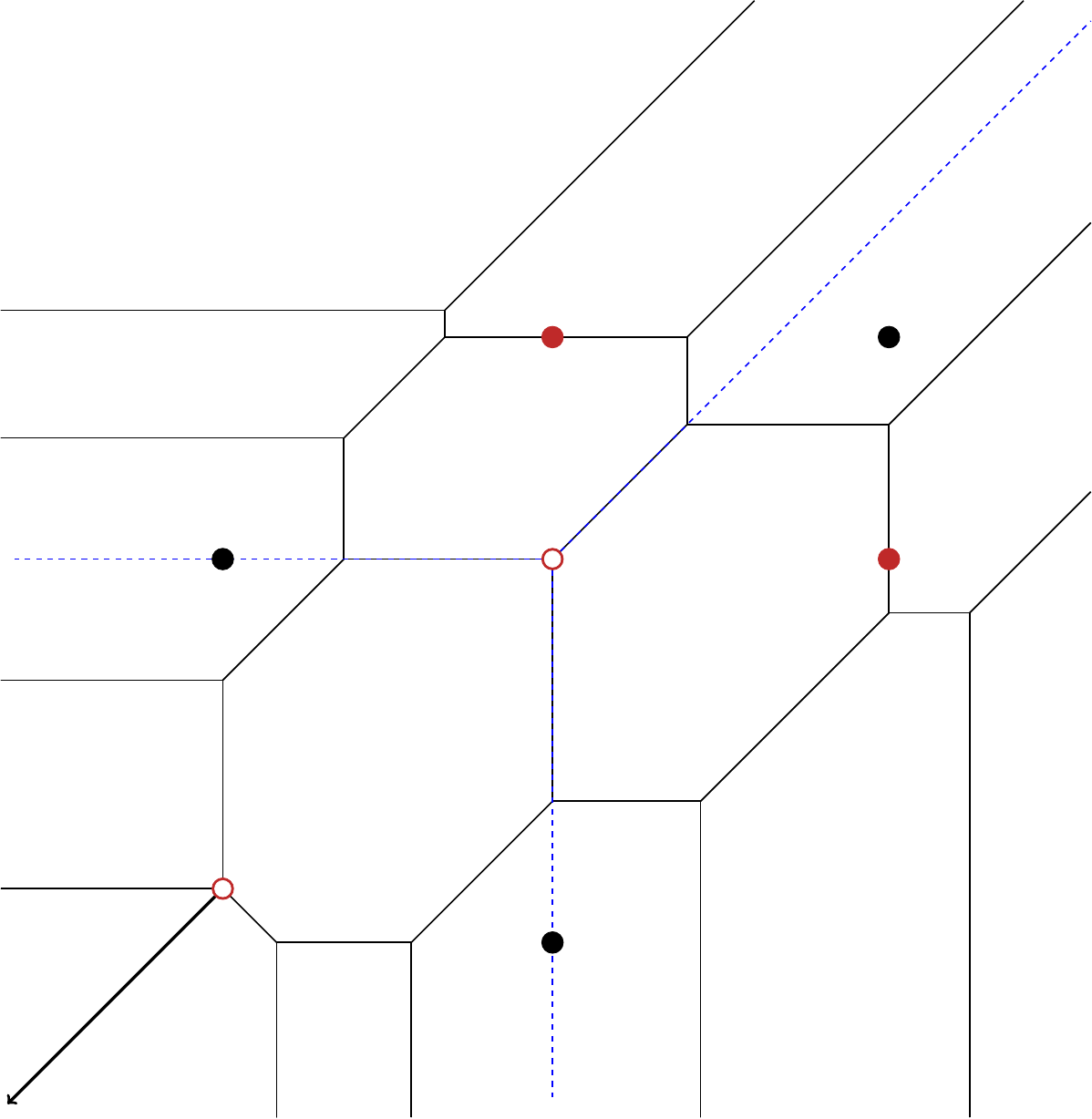}
		\caption{A smooth tropical quartic curve with its 7 bitangent classes. The bitangent class of shape (C) is depicted as a red circle filled with white. The corresponding tropical bitangent line is depicted in blue. }
		\label{fig:shapeC}
	\end{figure}
	By \cite{LeMa19} each tropical quartic curve has 7 bitangent classes.
	The real lifting conditions of the 7 bitangent classes depend only on the dual subdivision \cite{1GP21}.  They are sign conditions, i.e. given the dual subdivision and a distribution of signs on the lattice points, the number of real bitangents is determined. The sign conditions for all shapes of bitangent classes can be found in \cite[Table 11]{CueMa20}. In this article, we are mainly concerned with bitangent classes of shape (C) (see Figure \ref{fig:shapeC}) and their lifting behaviour. 
	\begin{proposition}[{\cite[Table 11, Corollary 7.3]{CueMa20}}]
		Let $\Gamma$ be a smooth, generic tropical quartic curve and $\delta$ a sign distribution on the lattice points of its dual subdivision $\mathcal{T}(\Gamma)$. A bitangent class of $\Gamma$ of shape (C) lifts over $\RR$ to $4$ totally real bitangents if and only if
		\begin{equation}
			(-\delta_{11})^{i+j}(\delta_{12})^i(\delta_{21})^j\delta_{0i}\delta_{j0} >0 \text{ and } (-\delta_{21})^{k+j}(\delta_{12})^k(\delta_{11})^j\delta_{k,4-k}\delta_{j0} >0. \label{eq:C}
		\end{equation}
	\end{proposition}
	
	The real lifting conditions of bitangents were determined by Cueto and Markwig under the following genericity conditions of the tropical quartic $\Gamma.$
	\begin{definition}[Genericity condition from {\cite[Remark 2.1]{CueMa20}}.]\label{def:generic} A smooth tropical quartic curve $\Gamma$ is called \emph{generic} if all its tropical bitangent lines are non-degenerate and the following three conditions are satisfied:
		\begin{itemize}
			\item[(i)]  if $\Gamma$ contains a vertex $v$ adjacent to three bounded edges with directions $-e_1$,
			$-e_2$ and $e_1 + e_2$, then the shortest of these edges is unique;
			\item[(ii)] A lift $V(f)$ of $\Gamma$ has no hyperflexes, i.e., no bitangent for which the two tangency points coincide;
			\item[(iii)]  the coefficients of $f$, with $V(f)$ a lift of $\Gamma$, are generic enough to guarantee that if the tangencies occur in the relative interior of the same end of $\Lambda$, then the local systems defined by these two points are inconsistent.
		\end{itemize}
	\end{definition}
	The first of these genericity conditions is of special interest when considering a bitangent class of shape (C), since the vertex in Definition \ref{def:generic} (i) is exactly the vertex of the tropical bitangent line of which a bitangent class of shape (C) consists. For an illustration, see Figure~\ref{fig:shapeC}.  We say that a tropical quartic curve is \emph{non-generic with respect to \ref{def:generic}(i)} if it is a smooth tropical curve satisfying conditions (ii) and (iii) of Definition \ref{def:generic}, but not (i).
	
	\begin{figure}
		\centering
		\includegraphics[width=0.175\linewidth]{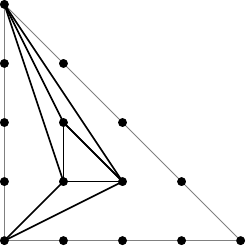}
		\caption{Any tropical quartic with dual subdivision containing this partial subdivision (up to $S_3$) is never generic.}
		\label{fig:nongeneric}
	\end{figure}
	Up to $S_3$ action, there are $8$ unimodular triangulations of $4\Delta_2$ such that any tropical quartic $\Gamma$ dual to such a triangulation can never satisfy condition (i) in Definition \ref{def:generic}. We call these unimodular triangulations \emph{non-generic}. They are refinements of the partial subdivision shown in Figure~\ref{fig:nongeneric}.

	A tropical curve that is non-generic w.r.t.$\,$\ref{def:generic}(i) can have a generic dual subdivision. For an example see Figure \ref{fig:honeycomb}. The locus of non-genericity in the secondary cones of generic tri-angulations was studied by the author in \cite{AlheydisThesis}. For completeness, we present the result here.
	\begin{figure}
		\centering
		\begin{subfigure}[c]{0.4\textwidth}\centering\includegraphics[width=0.5\linewidth]{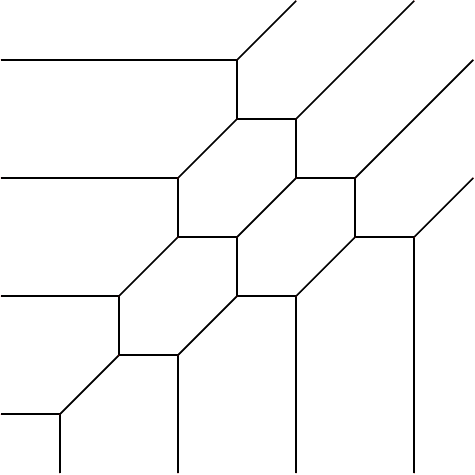}
		\end{subfigure}
		\begin{subfigure}[c]{0.4\textwidth}\centering\includegraphics[width=0.4\linewidth]{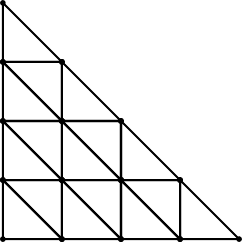}  
		\end{subfigure}
		\caption{The Honeycomb subdivision on the right is a generic subdivision. The tropical curve on the left is dual to this subdivision, but the curve does not satisfy genericity condition~\ref{def:generic}(i).}
		\label{fig:honeycomb}
	\end{figure}
	
	\begin{proposition}[{\cite[Proposition 6.3.4]{AlheydisThesis}}]\label{prop:Cnongenareas}
		A non-genericty locus of the secondary cone of a generic unimodular triangulation containing the motif of bitangent shape (C) is the intersection of the cone with at most three
		half-hyperplanes. These are shown in Table \ref{tab:Cnongenareas}.
	\end{proposition}
	\begin{table}[H]
		\centering
		\begin{tabular}{c}
			\hline
			$\mathbf{ \{	\boldsymbol\mu_3 > 	\boldsymbol\mu_2=	\boldsymbol\mu_1\}}$ \\
			\hline $\lambda_{21}-2\lambda_{11}+\lambda_{0i}-(i-1)(\lambda_{12}-\lambda_{11}) = \lambda_{12}-2\lambda_{11}+\lambda_{j0}-(j-1)(\lambda_{21}-\lambda_{11})$\\ 
			$(k-3)\lambda_{12}-(k-1)\lambda_{21}+\lambda_{k,4-k}+\lambda_{11}> \lambda_{12}-2\lambda_{11}+\lambda_{j0}-(j-1)(\lambda_{21}-\lambda_{11})$\\
			\hline
			\hline
			$\mathbf{\{	\boldsymbol\mu_2 > 	\boldsymbol\mu_1=	\boldsymbol\mu_3\}}$\\ \hline $\lambda_{21}-2\lambda_{11}+\lambda_{0i}-(i-1)(\lambda_{12}-\lambda_{11})=(k-3)\lambda_{12}-(k-1)\lambda_{21}+\lambda_{k,4-k}+\lambda_{11}$\\
			$\lambda_{12}-2\lambda_{11}+\lambda_{j0}-(j-1)(\lambda_{21}-\lambda_{11})>\lambda_{21}-2\lambda_{11}+\lambda_{0i}-(i-1)(\lambda_{12}-\lambda_{11})$\\
			\hline
			\hline
			$\mathbf{\{	\boldsymbol\mu_1 > 	\boldsymbol\mu_3=	\boldsymbol\mu_2\}}$\\
			\hline  $\lambda_{12}-2\lambda_{11}+\lambda_{j0}-(j-1)(\lambda_{21}-\lambda_{11}) = (k-3)\lambda_{12}-(k-1)\lambda_{21}+\lambda_{k,4-k}+\lambda_{11}$ \\
			$\lambda_{21}-2\lambda_{11}+\lambda_{0i}-(i-1)(\lambda_{12}-\lambda_{11}) > (k-3)\lambda_{12}-(k-1)\lambda_{21}+\lambda_{k,4-k}+\lambda_{11}$\	
		\end{tabular}
		\medskip
		\caption{The non-generic areas are given by intersecting the three half-hyperplanes as given above with the secondary cone that allows shape (C). Here, $\lambda_{ij}=\val(a_{ij})$ and $\mu_i$ is the lattice length of the edges in condition \ref{def:generic}(i) with direction vector $-e_i$, where $-e_3:= e_1+e_2$.}\label{tab:Cnongenareas}
	\end{table}

	Apriori, it is not clear that the lifting conditions for bitangents of shape (C) that were computed in \cite{CueMa20} hold for tropical quartics arising from these non-generic loci in the secondary cones. In the following, we see that the patchworking technique can solve this problem.
	We first turn towards tropical curves with a generic dual triangulation for which the coefficients of the curve are contained in a non-genericity locus as in Table \ref{tab:Cnongenareas}.

	\begin{proposition}\label{prop:Cliftsgen}
		Let $\mathcal{T}$ be a generic unimodular triangulation of $4\Delta_2$ that contains the dual motif of a bitangent class of shape (C). Then the real lifting conditions of bitangent class (C) are continuous over the interior of the secondary cone $\Sigma(\mathcal{T}).$ 
	\end{proposition}
	In other words, the lifting conditions hold for any smooth tropical curve dual to $\mathcal{T},$ even if the coefficients of this curve come from one of the non-genericity half-hyperplanes from Proposition \ref{prop:Cnongenareas}, i.e., if the curve does not satisfy the genericity condition \ref{def:generic}(i).
	\begin{proof}
		This follows directly from the patchworking principle: The real topology of an algebraic lift of the curve does not depend on the length of the edges in the tropical curve, but only on the combinatorial type. Thus, given a sign distribution, the number of real bitangents of the lift is determined for any element in the secondary fan. The lifting conditions for all other bitangent classes are independent of the genericity condition \ref{def:generic}(i). Therefore, they are the same when passing into the non-generic locus of the secondary cone. It follows that the lifting conditions for bitangents of shape (C) need to be consistent as well. Otherwise, the number of real bitangents could not be constant for the whole secondary cone.
	\end{proof}
	
	When considering real bitangents of a curve, one can always also ask whether the two tangency points are also real. If that is the case, the bitangent line is called \emph{totally real}. It is a result by Cueto and Markwig \cite{CueMa20} that real bitangents of generic smooth tropical quartic curves are always totally real. We see in the following, that we can extend this to smooth tropical curves which are non-generic w.r.t. \ref{def:generic}(i)  as long as the triangulation is generic.
	
	\begin{theorem}\label{thm:lift1}
		Let $\Gamma$ be a tropical quartic curve with a bitangent class of shape (C), dual to a generic unimodular triangulation. Assume that $\Gamma$ does not satisfy the genericity constraint, i.e. the tropical coefficients of $\Gamma$ lie in one of the regions of the secondary cone indicated by Table \ref{tab:Cnongenareas}.
		Then, whenever bitangent class (C) lifts to 4 real bitangents, these are totally real.
	\end{theorem}
	\begin{proof}
		Fix a lift $C$ of $\Gamma$ with a given sign-distribution $\delta$. Fix an approximating series of algebraic quartics $C_{\alpha}$ with the same sign distribution $\delta$ and the same topology, set $C_0=C.$
		We can choose these such that the following criteria are satisfied: they have the same dual triangulation $\mathcal{T}(\Gamma)$,  $\trop(C_{\alpha})$ converges to $\Gamma$, and all these curves satisfy the genericity constraint. By \cite{CueMa20} all real bitangents of $C_{\alpha}$, with $\alpha\neq 0$, are totally real. In particular, the real bitangents that tropicalize to the bitangent class of shape (C) are totally real.
		By \cite{LeMa19} the tangency points of the real bitangents are computed using the coefficients of the quartic and some of their radicands. Hence, if the coefficients of $C_{\alpha}$ converge to the coefficients of $C,$ so do the coordinates of the tangency points of the real bitangents.
		As $\RRt$ is closed, the bitangents of $C$ are all totally real, especially the ones tropicalizing to bitangent class~(C).
	\end{proof}
	Now we can turn to tropical curves which are dual to non-generic unimodular triangulations. The patchworking technique makes an algorithmic approach possible to determine the lifting condition for the missing bitangent class. This is formalized in Algorithm \ref{algo:Clifts}.
	\begin{algorithm}[h] \label{alg:one}
		\caption{Computing the lifting of bitangent shape (C)}
		\label{algo:Clifts}
		\begin{algorithmic}[1]
			\Require{Non-generic unimodular regular triangulation $\mathcal{T}$ of $4\Delta_2$.}
			\Ensure{All sign distributions for which bitangent class (C) lifts.}
			\ForEach{Sign vector $v\in\{\pm 1\}^{14}$} 
			\State Compute the number $n$ of real bitangents for all bitangent classes not of shape (C)
			\State Compute the Betti Number $b$ of the real patchworking curve 
			\If{$b=2$ }
			\State Compute whether the curve is dividing using twisted edges 
			\NoNumber{\textcolor{gray}{(if it is, the two ovals are nested)}}
			\EndIf
			\State Compute the number $N$ of real bitangents that the curve has by using the number of ovals $b$ and whether the curve is dividing.
			\If{$n\neq N$} 
			\State Collect $v$ as a sign vector for which the class of shape $(C)$ lifts.
			\EndIf
			\EndFor 
		\end{algorithmic}
	\end{algorithm}
	
	\begin{theorem}\label{thm:lift}
		Let $\Gamma$ be a smooth tropical quartic curve that does not satisfy the genericity condition \ref{def:generic}(i). Then $\Gamma$ has a bitangent class of shape (C) which lifts over $\RR$ if and only if the sign conditions in \eqref{eq:C} are satisfied. 
	\end{theorem}

	\begin{proof} If $\mathcal{T}(\Gamma)$ is a generic unimodular triangulations, this is Proposition \ref{prop:Cliftsgen}. For non-generic triangulations, this follows from \cite[Table 11]{CueMa20}, \cite{2GP21}, Proposition \ref{prop:dividing}, and the computation according to Algorithm \ref{algo:Clifts}. To be more precise, a non-generic triangulation has 6 bitangent classes for which we can compute the number of real lifts using \cite[Table 11]{CueMa20} and \cite{2GP21}. Since the number of real bitangents is uniquely determined by the topology of the real quartic curve \cite{Zeu73}, it is now possible to determine the real lifting of the seventh bitangent class, the one of type (C), by comparing the number of real lifts from the other 6 bitangent classes with the number of real bitangents expected to exist due to the number and types of ovals of the real curve. This is done by Algorithm \ref{algo:Clifts}. It uses Proposition \ref{prop:dividing} to determine the topology of the real curve by patchworking and twisted edges as explained in Section~\ref{sec:pre}. 
		The implementation of Algorithm \ref{algo:Clifts} in \texttt{polymake} uses the implementation of patchworking and real tropical hypersurfaces from Joswig and Vater \cite{JV20}.
	\end{proof}

	A consequence from the the computations in \cite{CueMa20} is that every real bitangent of a real lift of a generic smooth tropical quartic curve is totally real. We cannot say in general whether this still holds for bitangents of class (C) when the tropical quartic is dual to a non-generic unimodular triangulation, since the exact coefficients using tropical modifications cannot be determined in this case. However, we have the following result.
	
	\begin{proposition}\label{prop:totallyreal}
		For a smooth tropical quartic curve dual to a non-generic unimodular triangulation, at most one of the four real lifts from the bitangent class of shape (C) is not totally real.
	\end{proposition}
	\begin{proof}
		Klein's Formula connects the number of real inflection points $I$ with the number of real but not totally real bitangents $B$: $$I+2B = d(d-2).$$ By \cite[Section 5.4]{brugallé2012inflectionpointsrealtropical}, the count of real inflection points for any non-generic smooth triangulation away from the non-genericity yields $6$, and $d(d-2)=8.$ Hence, we conclude that for a smooth tropical quartic non-generic w.r.t. \ref{def:generic}(i) at most 1 real bitangent from the bitangent class of shape (C) is not totally real.   
	\end{proof}

	Using Theorem \ref{thm:lift}, we can now update the computational result on the distribution of numbers of real bitangents among the generic unimodular triangulations \cite[Theorem 4.2]{2GP21} to contain all unimodular triangulations up to symmetry. The data set with which we work is based on the original computation of all unimodular triangulations of $4\Delta_2$ from \cite{BJMS15} using \texttt{TOPCOM}~\cite{TOPCOM}.
	
	\begin{theorem} 
		The distribution of the number of real bitangents among the 1278 smooth unimodular triangulations
		of $4\Delta_2$ is reported in Table \ref{tab:my_label}. Every unimodular triangulation of $4\Delta_2$ admits a smooth tropical curve with a lift to
		a real plane curve with $28$ real bitangents.
		
		\begin{table}[H]
			\centering
			\begin{tabular}{ccccccc}
				$\{4\, 8\, 16\, 28\}$  & $\{4\, 8\, 28\}$ & $\{4\, 16\, 28\}$ & $\{8\, 16\, 28\}$ & $\{4\, 28\}$ & $\{8\, 28\}$ & $\{16 \,28\}$\\ \hline
				1207 & 15 & 27 &  18 &  6&  3 & 2   
			\end{tabular}
			\caption{Distribution of numbers of real bitangents among the dual subdivisions of smooth tropical quartic curves up to $S_3$ action.} 
			\label{tab:my_label}
		\end{table}
		
	\end{theorem}
	
	We close this section with the final count of real tropical quartics and their bitangents. 
	\begin{theorem}\label{thm:main}
		Table \ref{tab:main} shows the distribution of combinatorial types (Definition~\ref{def:real}) of real tropical quartic curves according to their number of real bitangents and real topological type: There are $\sim 28,7\%$ of combinatorial types of real tropical quartics with $1$ oval and $4$ real bitangents, $\sim 10,3\%$ with 2 nested ovals  and $4$ real bitangents, $\sim 32,3 \%$ with 2 non-nested ovals and $8$ real bitangents, $\sim 20,5\%$ with $3$ ovals and $16$ real bitangents, and $\sim 8,1\%$ with $4$ ovals and $28$ real bitangents.
	\end{theorem} 
	The proof of this theorem is purely computational. To achieve the results in Table \ref{tab:main} the computations to determine  the number of real bitangents and real topological type of the real tropical quartic were done for all $1278\cdot 2^{14}$ real tropical quartic curves up to $S_3$ symmetry using the updated \texttt{polymake} extension \texttt{TropicalQuarticCurves} accompanying this paper. The full count in the second to last line of Table \ref{tab:main} was achieved by multiplying each representative of the $S_3$ action with its orbit size. The data corroborating the count in Table \ref{tab:main} is available on \texttt{polyDB}.

	\section{Data and Implementation}\label{sec:ext}
	The documentation and availability of mathematical research data, in particular in combination with computational procedures and results, is becoming more important every day\footnote{\url{https://www.mardi4nfdi.de/}}~\cite{boege2022researchdatamanagementplanninggerman}. This project is a particular example of this, as it obtains new results by extending data and code from a previous project \cite{1GP21}, which itself was based on data computed for a different paper~\cite{BJMS15}.
	To allow for more such examples, the code and computed data from this project are made available to the research community according to the FAIR data principles~\cite{FAIR}.  This is particularly important when the original computation of the data is expensive in order to avoid recomputation. That is the case for the data of Theorem \ref{thm:main}, which was expensive in terms of computation time. Since the calculations were performed using \texttt{polymake}, the code is made available as an update to the already existing extension \texttt{TropicalQuarticCurves}. As medium for the computed data we use the database collection \texttt{polyDB}. The already existing collection on tropical quartic curves is extended by the new properties.

	\smallskip
	
	\paragraph{\texttt{polymake} extension.}
	The \polymake extension \texttt{TropicalQuarticCurves} presented in \cite{2GP21} is expanded by the new functions and the collection of tropical quartic curves on \texttt{polyDB} is updated to contain the full information on all non-generic triangulations, which were previously missing.
	The new version of the extension is available in the \texttt{polymake} wiki \footnote{\url{https://polymake.org/doku.php/extensions/tropicalquarticcurves}} and on GitHub \footnote{\url{https://github.com/AlheydisGeiger/TropicalQuarticCurves-0.2}}.

	In the updated version, you can now
	\begin{itemize}
		\item compute the number of real bitangents for any unimodular triangulation of $4\Delta_2$ and a given sign vector,
		\item obtain the sign conditions for any unimodular triangulation of $4\Delta_2$ also for the bitangent class of shape (C),
		\item work with a real tropical curve and determine whether it is dividing, compute its twisted edges and number of ovals.
	\end{itemize}
	In the following we showcase the new objects and commands in the extension for each application and the database.
	\medskip

	\textit{Application tropical.} 
	We introduce a new object \verb|RealQuarticCurve|, which consists of a tropical quartic curve together with a sign distribution. 
	\begin{lstlisting}
tropical > $C = new RealQuarticCurve<Min>(COEFFICIENTS=>[5,1,2,2,0,0,4,0,1,16,7,9,12,16,33],SIGNS=>[1,1,1,1,1,1,1,1,1,1,1,1,1,1,1]);	
	\end{lstlisting}
	One can still access the tropical quartic curve,
	as well as its properties like \verb|BITANGENT_SHAPES|, which gives the shapes of the bitangent classes for the given point in the secondary fan of $\mathcal{T}(\Gamma).$
	\begin{lstlisting}
tropical > $Q= $C->QUARTIC_CURVE;
tropical > print $C->BITANGENT_SHAPES;
A H` J J G T W
	\end{lstlisting}
	The tropical bitangent classes that lift for the given sign distribution can be accessed as follows:
	\begin{lstlisting}
tropical > print $C->N_REAL_BITANGENTS;
4
tropical > print $C->REAL_BITANGENT_TYPES;
EFJ G GKUTT` W...HH+(xz)
tropical > print $C->REAL_BITANGENT_SHAPES;
J G T W
tropical > $Bitangentclasses = $C->REAL_BITANGENTS;
		
tropical > for my $i (0..$C->N_REAL_BITANGENTS-1) {
tropical (2)> print $Bitangentclasses->[$i]->SIGN_CONDITIONS;
tropical (3)> }
{5 10}
{}
{5 10}
{}
{5 10}
{}
{}
{}
	\end{lstlisting}
	
	Further, we can decide whether the real curve is dividing and how many ovals an real algebraic lift would have.
	\begin{lstlisting}
tropical > print $C->IS_DIVIDING;
false
tropical > print $C->N_OVALS;
3
	\end{lstlisting}
	For a curve with two ovals, this would tells us whether or not the two ovals are nested: If the curve is not dividing, the two ovals are not nested.
	To gain a picture, you can access the patchworking functionality of \texttt{polymake} as below. The resulting visualization is depicted in Figure \ref{fig:patch}.
	\begin{lstlisting}
tropical > $p= $C->PATCHWORK;
tropical > $p->realize->VISUAL;
	\end{lstlisting}   
	
	\begin{figure}
		\centering
		\includegraphics[width=0.35\linewidth]{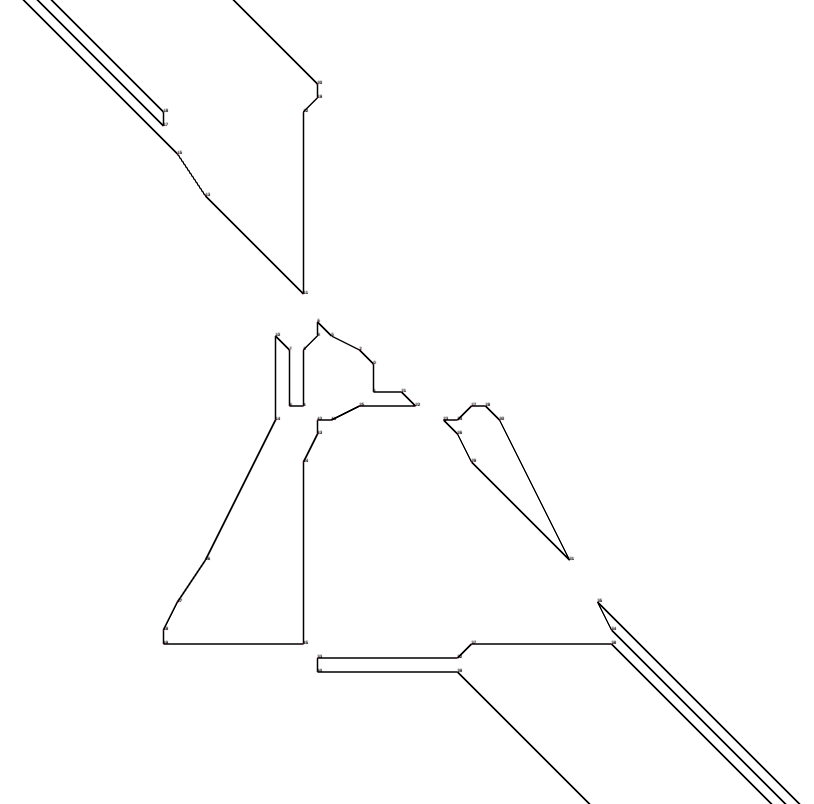}
		\caption{The picture from patchworking for the curve \texttt{\$C}.}
		\label{fig:patch}
	\end{figure}
	
	The property \verb|TWISTED_EDGES| records a list of edges in the dual subdivision of the tropical curve that is twisted with respect to the given sign distribution. It is imperative to note, that these are not edges of the curve \verb|$C|, but of its dual subdivision \verb|$C->DUAL_SUBDIVISION|.
	\begin{lstlisting}
tropical > print $C->TWISTED_EDGES;
The twisted edges in the dual subdivision are
(8 10) (5 7) (5 8) (5 13) (4 7) (4 5) (1 7) (1 5) (1 4)
	\end{lstlisting}

	\textit{Application fan.} There are few changes in this application compared to the first version of the extension \texttt{TrpicalQuarticCurves}. Since the twisted edges as well as the property to be dividing depend only on the combinatorial type, these properties can be directly computed for the underlying subdivision without the need to fix an actual tropical curve. 
	\begin{lstlisting}
fan > $S = new DualSubdivisionOfQuartic(MAXIMAL_CELLS=>[[0,1,2],[1,2,4], [2,4,12],[4,7,12],[2,8,12],[2,8,13],[8,12,13],[2,5,13],[5,9,13], [9,13,14],[7,11,12],[7,10,11],[4,7,10],[4,6,10],[3,4,6],[1,3,4]]);	
fan > $v = new Vector<Int>([1,1,1,1,1,1,1,1,1,1,1,1,1,1,1,1]);
fan > print $S->IS_DIVIDING($v);
false
fan > print $S->TWISTED_EDGES($v);
(8 13) (8 12) (4 10) (2 4) (2 12) (2 13) (2 8) (1 4) (1 2)
	\end{lstlisting}

	The methods to determine the twisted edges and whether the curve is dividing are based on Theorems from \cite{BIMS15,texier2021topologyrealalgebraiccurves} which hold more generally for any subdivision. Therefore, we included functions which are available for any \texttt{SubdivisionOfPoints}. That broadens the options of their possible application by other researchers.
	
	\begin{lstlisting}
fan > $T = new SubdivisionOfPoints(POINTS=>[[1,0,0],[1,1,0],[1,0,1],
[1,2,0],[1,1,1],[1,0,2]],MAXIMAL_CELLS=>[[0,1,2],[1,2,4],[2,4,5],
[2,3,4]]);
fan > $v = new Vector<Int>([1,1,1,1,1,1]);
fan > print twisted_edges($T,$v);
(1 2)
fan > print is_dividing($T,$v);
true
	\end{lstlisting}
	
	\paragraph{Database \texttt{polyDB}.}  The computation that was performed in order to achieve Theorem \ref{thm:main} was time expensive: For example, for the randomly chosen triangulation in the database with identifier \#766 the computation of \verb|OVALS| took 4408 seconds on a MacBook Air with Sequoia 15.3.1, Chip Apple M3 and 24 GB Memory. 
	The big advantage of a database is that all these computations can now be accessed very quickly and do not have to be computed again. Therefore, the collection \texttt{Tropical.QuarticCurves} on \texttt{polyDB} was extended as described below. The database can also be used for examples not equal to the $S_3$ representatives stored there:
	Given a unimodular triangulation $\mathcal{T}$ of $4\Delta_2$, it is possible to obtain the database identifier of the unique triangulation within collection \texttt{Tropical.QuarticCurves} that is in the $S_3$-orbit of $\mathcal{T}$ by using the function \verb|find_in_database|.
	
	\begin{lstlisting}
fan > $polydb = polyDB();
fan > $collection = $polydb->get_collection("Tropical.QuarticCurves");
fan > $S = new DualSubdivisionOfQuartic(WEIGHTS=>[6,3,1,1,0,0,3,0,1,3,14,10,8,7,7]);
fan > print find_in_database($S);
766
	\end{lstlisting}
	
	We used the original data from \cite{1GP21} and updated it according to the results in Section \ref{sec:results}. The updated data now additionally contains: 
	\begin{itemize}
		\item for non-generic triangulations
		\begin{itemize}
			\item \verb|PLUECKER_NUMBERS| the number of possible real bitangents,  
			\item \verb|SIGN_REPRESENTATIVES| an exemplary sign vector for each number of real bitangents,
		\end{itemize}
		\item the sign conditions for the bitangent classes of shape (C),
		\item the property \verb|OVALS| which contains the data that led to Theorem \ref{thm:main}, i.e., the count of the number of sign vectors for which the triangulation has a given number of ovals together with an exemplary sign vector.
	\end{itemize}
	The following shows the new property \verb|OVALS| in more detail for a downloaded file from the data base \texttt{497.json}.
	\begin{lstlisting}
fan > $S = load_data("PATH/497.json");
fan > print $S->OVALS->COUNT;
4096 6144 4096 1024
fan > print $S->OVALS->SIGNS;
1 1 1 1 1 1 1 1 1 1 1 1 1 1 1
1 1 -1 1 1 1 1 1 1 1 1 1 1 1 1
1 1 1 1 -1 1 1 1 1 1 1 1 1 1 1
1 1 1 1 1 1 1 1 1 1 -1 1 1 1 1
	\end{lstlisting}
	The vector \verb|$S->OVALS->COUNT| denotes the number of real tropical curves with the dual subdivision \verb|$S| that have $1$, $2$, $3$ or $4$ non-nested ovals as their real topology, respectively. The property \verb|$S->OVALS->SIGNS| stores representative sign vectors that provide the corresponding real topology. The case of $2$ nested ovals can be accessed separately.
	\begin{lstlisting}
fan > print $S->OVALS->COUNT_NESTED;
1024
fan > print $S->OVALS->SIGNS_NESTED;
1 1 1 1 1 1 1 -1 1 1 1 1 1 1 1
	\end{lstlisting}
	The complete data can be viewed in a collected form using the method $\verb|OVALS_INFORMATION|$.
	\begin{lstlisting}
fan > print $S->OVALS_INFORMATION;
(((1 not nested) 4096) <1 1 1 1 1 1 1 1 1 1 1 1 1 1 1>)
(((2 not nested) 6144) <1 1 -1 1 1 1 1 1 1 1 1 1 1 1 1>)
(((3 not nested) 4096) <1 1 1 1 -1 1 1 1 1 1 1 1 1 1 1>)
(((4 not nested) 1024) <1 1 1 1 1 1 1 1 1 1 -1 1 1 1 1>)
(((2 nested) 1024) <1 1 1 1 1 1 1 -1 1 1 1 1 1 1 1>)
	\end{lstlisting}
	
	The database can also be searched for all triangulations with a given property. 
	\begin{lstlisting}
fan > $polydb = polyDB();
fan > $collection = $polydb->get_collection("Tropical.QuarticCurves");
fan > print collection->count({"OVALS.COUNT.3" => 1024});
442
	\end{lstlisting}
	This means that there are $442$ subdivisions $\mathcal{T}$ for which the associated real tropical curve $(\Gamma,\delta)$ has 4 ovals for exactly $1024$ sign distributions $\delta$.
	
	\bibliographystyle{plain}
	\bibliography{main}{}
	
\end{document}